\documentclass{article}
\usepackage{amsthm}
\usepackage{amsmath}
\usepackage{amssymb}

\title{Local rigidity of homogeoenous actions of parabolic subgroups
 of rank-one Lie groups}
\author{Masayuki ASAOKA
\footnote{Partially supported by JSPS Grant-in-Aid for Young
Scientists (B), No.19740085.}\\
Department of Mathematics, Kyoto University
}

\def\RR{\mathbb{R}}

\def\cF{{\cal F}}

\def\cA{{\cal A}}

\def\SO{{SO_+(n,1)}}
\def\so{{\mathfrak{so}(n,1)}}
\DeclareMathOperator{\Fr}{Fr}

\def\bs{{\backslash}}

\def\ra{{\rightarrow}}
\def\st{{\;|\;}}

\def\hsp{{\hspace{3mm}}}

\theoremstyle{plain}
\newtheorem{thm}{Theorem}[section]
\newtheorem{prop}[thm]{Proposition}
\newtheorem{lemma}[thm]{Lemma}

\theoremstyle{definition}

\theoremstyle{remark}

\begin{document}
\maketitle

\begin{abstract}
We show the local rigidity of the natural action
 of the Borel subgroup of $SO_+(n,1)$
 on a cocompact quotient of $SO_+(n,1)$
 for $n \geq 3$.
\end{abstract}

\section{Introduction}
\label{sec:introduction}

%
%   Introduction
%
%\subsection{Local rigidity of actions}
Rigidity theory of actions of non-compact groups
 has been rapidly developed in the last two decades.
It is found that
 many actions related to Lie groups of real-rank greater than one
  exhibit rigidity.
%\cite{FM09,Ka96,KS97,KLZ96,MQ01,NT95,NT01},
See Fisher's survey paper \cite{F07}, for example.
However, there are only few results on actions
 related to Lie groups of real-rank one.
The aim of this paper is to show the local rigidity
 of some natural actions related to such groups.

Let $G$ be a Lie group and $M$ be a $C^\infty$ manifold.
By $C^\infty(M \times G,M)$,
 we denote the space of $C^\infty$ maps from $M \times G$ to $M$
 with the compact-open $C^\infty$-topology.
Let $\cA(M,G)$ be the set of $C^\infty$ right actions of $G$ on $M$.
It is a closed subset of $C^\infty(M \times G,M)$.
We say two actions $\rho_1:M_1 \times G \ra M_1$
 and $\rho_2:M_2 \times G \ra M_2$ are {\it $C^\infty$-conjugate}
 if there exists a $C^\infty$ diffeomorphism $h$
 and an automorphism $\sigma$ of $G$ such that
 $h(\rho_1(x,g))=\rho_2(h(x),\sigma(g))$
 for any $x \in M_1$ and $g \in G$.
An action $\rho \in \cA(M,G)$ is  called
 {\it $C^\infty$-locally rigid} if
 the $C^\infty$-conjugacy class of $\rho$
 is a neighborhood of $\rho$ in $\cA(M,G)$.
We say an action $\rho \in \cA(M,G)$ is {\it locally free}
 if the isotropy subgroup
 $\{g \in G \st \rho(x,g)=x\}$
 is a discrete subgroup of $g$ for any $x \in M$.

Let $H$ be its closed subgroup of a Lie group $G$
 and $\Gamma$ a cocompact lattice of $G$.
We define {\it the standard $H$-action} $\rho_0$ on $\Gamma \bs G$
 by $\rho_0(\Gamma g,h) =\Gamma(gh)$.
It is a locally free action.
We say an action is {\it homogeneous} if it is $C^\infty$-conjugate
 to the standard action associated with some cocompact lattice.

Suppose that the Lie group $G$ is connected and semi-simple.
Let $G=KAN$ be its Iwasawa decomposition.
The dimension of the abelian subgroup $A$ is called
 {\it the real-rank} of $G$.
Let $M$ be the centralizer of $A$ in $K$.
The group $P=MAN$ is called {\it the Borel subgroup} associated
 with the Iwasawa decomposition $G=KAN$.
It is known that the conjugacy class of the Borel subgroup
 does not depend  on the choice of the Iwasawa decomposition.
When $G=SL(2,\RR)$ for example,
 a Borel subgroup $P$ is conjugate to
 the group $GA$ of the upper triangular matrices
 in $SL(2,\RR)$.\footnote{It is isomorphic to
 the \underline{G}roup of orientation preserving
 \underline{A}ffine transformations of the real line.}
Fix a cocompact lattice $\Gamma$ of $SL(2,\RR)$ and
 put $M_\Gamma=\Gamma \bs SL(2,\RR)$.
Let $\rho_0$ be the standard $P$-action on $M_\Gamma$.
It is not locally rigid
 since deformation of lattice $\Gamma$
 gives a non-trivial deformation of actions.
However, there are several rigidity results on $\rho_0$
Ghys \cite{Gh79} proved that if a locally free $GA$-action on $M_\Gamma$
 admits an invariant volume, then it is homogeneous.
In \cite{Gh85}, he remove the assumption on invariant volume
 when $H^1(M_\Gamma)$ is trivial.
The author of this paper \cite{As07} classified
 all locally free action of $P$ on $M_\Gamma$
 up to $C^\infty$-conjugacy without any assumption.
As a consequence,
 there exists a non-homogeneous locally free $GA$-action
 on $M_\Gamma$ when $H^1(M_\Gamma)$ is non-trivial.
In a forthcoming paper, he will also
 show that the standard $GA$-action $\rho_0$ admits
 a $C^\infty$ deformation into non-homogeneous actions in this case.

%When $G$ is $SL(2,\CC)$,
% the Borel subgroup $P$ is conjugate to
% the group $GA_\CC$ of upper traiangle matrices, similar to $SL(2,\RR)$.
%In \cite{GV94},
% Ghys and Verjovsky showed that
% if a holomorphic locally free action of $GA_\CC$
% admits a $C^0$-invariant volume, then it is homogeneous.
%Ghys \cite{Gh95} also classified `holomorphic Anosov systems'
% on complex three-dimensional manifolds.
%As a consequence of his result,
% any holomorphic action sufficiently $C^1$-close to $\rho_\Gamma$
% is holomorphically conjugate to $\rho_\Gamma$.

By Mostow's rigidity theorem,
 any deformation of a cocompact lattice is trivial
 if $G$ is a higher-dimensional Lie group of real-rank one.
So, it is natural to ask
 whether the standard $P$-action
 is $C^\infty$-locally rigid or not in this case.
The main result of this paper answers this question
 when $G$ is $SO_+(n,1)$.
\begin{thm}
\label{thm:main thm} 
Let $P$ be a Borel subgroup of $SO_+(n,1)$
 and $\Gamma$ be a torsion-free cocompact lattice of $SO_+(n,1)$.
If $n \geq 3$, then the standard $P$-action on $\Gamma \bs SO_+(n,1)$
 is $C^\infty$-locally rigid.
\end{thm}

To ending the introduction,
 we remark on the local rigidity of the orbit foliation.
Let $\cF_\Gamma$ be the orbit foliation
 of the standard $P$-action on $\Gamma \bs \SO$.
Ghys \cite{Gh93} showed a global rigidity result
 of $\cF_\Gamma$ for $n=2$.For $n \geq 3$,
Yue \cite{Yu95} proved a partial result
 and Kanai \cite{Ka99} claimed the local rigidity of $\cF_\Gamma$.
However, Kanai's proof contains a serious gap\footnote{
The $C^1$-regularity of the strong unstable foliation
 claimed in the last sentence of p.677 does not hold in general.}
 and it is not fixed so far.
Hence, the local rigidity of $\cF_\Gamma$ is still open.
If any foliation sufficiently close to $\cF_\Gamma$
 carries an action of $P$,
 then the local rigidity of $\cF_\Gamma$ follows
 from our theorem.

%
% Preliminaries
%

\section{Preliminaries}
%
%   Preliminaries
%

In this section,
 we introduce some notations
 and review several known facts which we will use 
 in the proof of Theorem \ref{thm:main thm}.

\subsection{The group $SO_+(n,1)$}
\label{sec:SO}
Fix $n \geq 3$
 and let $I_{n,1}$ be the diagonal matrix of size $(n+1)$
 whose diagonal elements are $1,\dots,1,-1$.
Let $\SO$ be the identity component of the subgroup
 of $GL(n+1,\RR)$ consisting of matrices $A$
  satisfying $^\tau\hspace{-1mm} A I_{n,1} A =I_{n,1}$.
For any $1 \leq n' \leq n$,
 the standard embedding $GL(n',\RR) \hookrightarrow GL(n+1,\RR)$
 induces an embedding of $SO(n')$ into $\SO$.

Let $\so$ be the Lie algebra of $\SO$.
By $E_{ij}$, we denote the square matrix of size $(n+1)$ 
 such that the $(i,j)$-entry is one and other entries are zero.
Put $X=E_{n(n+1)}-E_{(n+1)n}$,
 $Y_i=(E_{i(n+1)}+E_{(n+1)i})-(E_{in}-E_{ni})$, and
 $Y'_i=(E_{i(n+1)}+E_{(n+1)i})+(E_{in}-E_{ni})$
 for $i,j=1,\cdots,n$.
Then, $\so$ is generated by
 $X,Y_1,\cdots,Y_{n-1},Y'_1,\cdots,Y'_{n-1}$
 and the Lie subalgebra corresponding to the subgroup $SO(n-1)$ of $\SO$.
It is easy to check that
\begin{gather}
\label{eqn:Y 1}
[Y_i,X]  =-Y,\hsp [Y'_i,X]  =Y'_i,\hsp
[Y_i,Y_j]  = [Y'_i,Y'_j]  =0
\end{gather}
 for any $i,j=1,\cdots,n-1$, and
\begin{gather}
\mbox{Ad}_m(X)=X \nonumber\\
\label{eqn:Y 2}
 \mbox{Ad}_m(Y_1,\cdots,Y_{n-1})
 =(Y_1,\cdots,Y_{n-1}) \cdot m
\end{gather}
 and for any $m \in SO(n-1)$.

Let $\SO=KAN$ be 
the Iwasawa decomposition of $\SO$
 associated with the involution
 $\theta_0: g \mapsto (^\tau g)^{-1}$.
Then, $K=SO(n)$ and,
 $A$ and $N$ are the subgroups of $\SO$
 corresponding to the Lie subalgebras
 spanned by $X$ and $\{Y_1,\dots Y_{n-1}\}$, respectively.
Since the centralizer of $A$ in $SO(n)$ is $SO(n-1)$,
 the Borel subgroup corresponding to $\theta_0$ is $SO(n-1)AN$.

\subsection{Anosov flows}
A $C^1$ flow $\Phi$ on a closed manifold $M$ is called {\it Anosov},
 if it has no stationary points
 and there exists a continuous splitting
 $TM=T\Phi \oplus E^{ss} \oplus E^{uu}$, a constant $\lambda>0$,
 and a continuous norm $\|\cdot\|$ on $TM$ which satisfy
 the following properties:
\begin{itemize}
\item $T\Phi$ is
 the one-dimensional subbundle tangent to the orbit of $\Phi$.
\item $E^{ss}$ and $E^{uu}$ are $D\Phi$-invariant subbundles.
\item $\|D\Phi^t(v^s)\| \leq e^{-\lambda t}\|v^s\|$ and
 $\|D\Phi^t(v^u)\| \geq e^{\lambda t}\|v^u\|$
 for any $v^s \in E^{ss}$, $v^u \in E^{uu}$, and $t \geq 0$.
\end{itemize}
The subbundles $E^{ss}$, $E^{uu}$,
 $T\Phi \oplus E^{ss}$, and $T\Phi \oplus E^{uu}$ are called
 {\it the strong stable, strong unstable,
 weak stable}, and
 {\it weak unstable subbundles}, respectively.
It is known that they generate continuous foliations
 with $C^r$ leaves, if $\Phi$ is a $C^r$ flow.
The foliations are called {\it the strong stable foliation}, {\it etc.}

The following proposition
 may be well-known for experts, but we give a proof
 for convenience of the readers.
\begin{prop}
\label{prop:fol conj}
Let $\Phi_1$ and $\Phi_2$ be Anosov flows on a closed manifold $M$.
Suppose that  $\Phi_1$ and $\Phi_2$ have
 the common strong unstable foliation $\cF^{uu}$
 and $\cF^{uu}(\Phi_1^t(x)) = \cF^{uu}(\Phi_2^t(x))$
 for any $x \in M$ and $t \in \RR$.
Then, there exists a homeomorphism $h$ of $M$ such that
 $\Phi_2^t \circ h=h \circ \Phi_1^t$ for any $t \in \RR$
 and $h(\cF^{uu}(x))=\cF^{uu}(x)$ for any $x \in M$.
\end{prop}
\begin{proof}
Let ${\cal H}$ be the set of continuous maps $h:M \ra M$ which
 preserves each leaf of $\cF^{uu}$.
Fix a Riemannian metric $g$ of $M$.
Let $d^u$ be the leafwise distance on leaves of $\cF^{uu}$
 which is determined by the restriction of the metric $g$ to each leaf.
We define a distance $d$ on ${\cal H}$ by
 $d(h,h')=\sup_{x \in N_\Gamma}d^u_\Gamma(h(x),h'(x))$.
It is a complete metric on ${\cal H}$.

For $i,j \in \{1,2\}$,
 we define continuous flows $\Theta_{ij}$ on ${\cal H}$ by
 $\Theta_{ij}^t(h)=\Phi_i^{-t} \circ h \circ \Phi_j^t$.
Since $\Phi_i$ and $\Phi_j$ expand $\cF^{uu}$ uniformly,
 $\Theta_{ij}^t$ is a uniform contraction
 for any sufficiently large $t>0$.
By the contracting mapping theorem,
 there exists a unique fixed point $h_{ij} \in {\cal H}$ of
 the flow $\Theta_{ij}$.
Since both $h_{ij} \circ h_{ji}$ and the identity map of $M$
 are fixed point of $\Theta_{ii}$,
 $h_{ij} \circ h_{ji}$ is the identity map for $i,j \in {1,2}$.
In particular, $h_{ij}$ is the inverse of $h_{ji}$.
Therefore, $h_{21}$ is a homeomorphism in ${\cal H}$
 such that
 $h_{21} \circ \Phi_1^t=\Phi_2^T \circ h_{21}$ for any $t \in \RR$.
\end{proof}

Let $\Psi$ be a flow on a manifold $M$.
A $C^\infty$ function $\alpha$ on $M \times \RR$
 is {\it a cocycle over $\Psi$} if
 $\alpha(x,0)=0$ and $\alpha(x,t+t')=\alpha(x,t)+\alpha(\Psi^t(x),t')$
 for any $x \in M$ and $t,t' \in \RR$.
We say $\Psi$ is {\it topologically transitive}
 if there exists $x_0 \in M$
 whose orbit $\{\Psi^t(x_0) \st t \in \RR\}$ is a dense subset of $M$.
\begin{thm}
[The $C^\infty$ Livschitz Theorem \cite{LMM}] 
\label{thm:Livschitz}
Let $\Phi$ be a $C^\infty$ topologically transitive Anosov flow on 
 a closed manifold $M$
 and $\alpha$ be a $C^\infty$ cocycle over $\Phi$.
If $\alpha(x,T)=0$ for any $(x,t) \in M \times \RR$
 satisfying $\Phi^T(x)=x$,
 then there exists a $C^\infty$ function $\beta$ on $M$
 such that $\alpha(x,t)=\beta(\Phi^t(x))-\beta(x)$
 for any $x \in M$ and $t \in \RR$.
Moreover, if $\alpha$ is sufficiently $C^\infty$-close to $0$,
 then we can choose $\beta$ so that it is $C^\infty$-close to $0$.
\end{thm}

We say an Anosov flow $\Phi$ is {\it $s$-} ({\it resp.} $u$-){\it conformal}
 if $D\Phi^t$ is conformal on $E^{ss}(x)$ ({\it resp. }$E^{uu}(x)$)
 for any $x \in M$
 with respect to some continuous metric\footnote{We may assume that
 the metric is H\"older continuous
 and $C^\infty$ along leaves of the strong stable foliation.
 See Sadovskaya \cite{Sa05}.} on $E^{ss}$.
The following result plays fundamental role
 in the proof of Theorem \ref{thm:main thm}.
\begin{thm}
[de la Llave \cite{Ll}] 
\label{thm:Llave}
Let $\Phi_1$ and $\Phi_2$ be $C^\infty$ $s$-conformal
 topologically transitive Anosov flows on a closed manifold $M$.
For $i=1,2$, let $\cF^{ss}_i$ be  of $\Phi_1$.
Suppose that the dimensions of the strong stable foliation
 of $\Phi_1$ and $\Phi_2$ are greater than one.
If a homeomorphism $h$ of $M$ satisfies
 $\Phi_2^t \circ h = h \circ \Phi_1^t$ for any $t \in \RR$,
 then the restriction of $h$ to a leaf of
 the strong stable foliation of $\Phi_1$
 is a $C^\infty$ diffeomorphism to a leaf of
 the strong stable foliation of $\Phi_2$.
Moreover,  if both $\Phi_1$ and $\Phi_2$ are $u$-conformal in addition,
 then $h$ is a $C^\infty$ diffeomorphism of $M$.
\end{thm}

We say that an Anosov flow is {\it contact}
 if it preserves a $C^1$-contact structure.
It is easy to see that
 any contact structure invariant under an Anosov flow
 is the direct sum of the strong stable subbundle
 and the strong unstable subbundle.
\begin{prop}
\label{prop:contact Anosov} 
Let $\Phi$ be a contact Anosov flow on a closed manifold $M$.
If $\Phi$ is $s$-conformal, then it is $u$-conformal.
\end{prop}
\begin{proof}
Let $TM=T\Phi \oplus E^{ss} \oplus E^{uu}$
 be the Anosov splitting of $\Phi$
 and $X$ be the vector field generating $\Phi$.
Take a continuous metric $g_-$ such that $\Phi$ is $s$-conformal
 with respect to $g_-$.
Let $\alpha$ be one-form $\alpha$ on $M$ such that
 $\mbox{Ker}\, \alpha=E^{ss} \oplus E^{uu}$ and $\alpha(X)=1$.
Since $E^{ss} \oplus E^{uu}$ is a $\Phi$-invariant $C^1$-contact structure,
 the one-form $\alpha$
 is a $C^1$-contact form invariant under $\Phi$.
Hence, $\omega=d\alpha$ is a $D\Phi$-invariant
 two-form which is non-degenerate on $E^{ss} \oplus E^{uu}$.

By the invariance,
 $\omega(v,v')=\omega(D\Phi^t(v),D\Phi^t(v'))$ 
 for $x \in M$, $v,v' \in E^{ss}(x)$, and $t \in \RR$,
 and the latter converges to zero as $t \ra +\infty$.
Hence, the restriction of $\omega$ to $E^{ss}$ is zero.
The restriction of $\omega$ to $E^{uu}$ also is.

We define a metric $g_+$ on $E^{uu}$ by
\begin{equation*}
 g_+(v,v')=\sum_{j=1}^n\omega(u_i,v)\omega(u_i,v')
\end{equation*}
 for $v,v' \in E^{uu}(x)$,
 where $(u_1,\dots,u_n)$ is an orthonormal basis
 of $E^{ss}(x)$ with respect to $g_-.$
By a direct calculation,
 we can check that $g_+$ does not depend on the choice
 of the orthonormal basis $(u_1,\dots,u_n)$.
Remark that $g_+$ is a continuous metric on $E^{uu}$.

Fix $x \in M$ and $t \in \RR$.
Since $\Phi$ is $s$-conformal,
 there exists a real number $a \neq 0$
 and orthonormal basis $(u_1,\dots,u_n)$ of $E^{ss}(x)$
 and $(u'_1,\dots,u'_n)$ of $E^{ss}(\Phi^t(x))$
 such that $D\Phi^t(u_i) =a \cdot u'_i$ for any $i=1,\cdots,n-1$.
Then, 
\begin{equation*}
 \omega(u'_i,D\Phi^t(v))
 =a^{-1}\omega(D\Phi^t(u_i),D\Phi^t(v))=a^{-1}\omega(u_i,v)
\end{equation*}
 for any $v \in E^{uu}(x)$.
It implies that
 $g_+(D\Phi^t(v),D\Phi^t(v'))=a^{-2} g_+(v,v')$
 for any $v,v' \in E^{uu}(x)$.
Therefore, $D\Phi$ is $u$-conformal with respect to $g_+$.
\end{proof}

%
%  Proof
%

\section{Proof of  Theorem \ref{thm:main thm}}
Let $\SO=SO(n)AN$ be the Iwasawa decomposition
 and $P=SO(n-1)AN$ the Borel subgroup of $\SO$
 described in Section \ref{sec:SO}.
Fix a torsion-free cocompact lattice $\Gamma$ of $\SO$,
 and put $M_\Gamma =\Gamma \bs \SO$ and
 $N_\Gamma =\Gamma \bs \SO /SO(n-1)$.
We denote the natural projection
 from $M_\Gamma$ to $N_\Gamma$ by $\pi$.
Let $\rho_0$ be the standard $P$-action on $M_\Gamma$.
For $x=\Gamma g \in M_\Gamma$ and $m \in SO(n-1)$,
 we put $x \cdot m=\rho_0(x,m)=\Gamma(gm)$.

For any $\rho \in \cA(M_\Gamma,P)$ and $g \in P$,
 we define a $C^\infty$ diffeomorphism $\rho^g$ of $M_\Gamma$
 by $\rho^g(x)=\rho(x,g)$.

\subsection{Induced Anosov flows on $N_\Gamma$}
Let $\cA_*(M_\Gamma,P)$ be the set of locally free $P$-actions on
 $M_\Gamma$ which satisfy $\rho(x,m)=x \cdot m$
 for any $x \in M_\Gamma$ and $m \in SO(n-1)$.
\begin{prop}
\label{prop:Palais}
If $\rho:M_\Gamma \times P \ra M_\Gamma$ is
 sufficiently $C^\infty$-close to $\rho_\Gamma$
 then $\rho$ is $C^\infty$ conjugate to an action in $\cA_*(M_\Gamma,P)$
 which is $C^\infty$-close to $\rho_\Gamma$.
\end{prop}
\begin{proof}
It is an immediate corollary of Palais' stability theorem of
 compact group action (\cite{Pa61}).
\end{proof}

For $\rho \in \cA_*(M_\Gamma,P)$,
 we define a flow $\Phi_\rho$ on $N_\Gamma$
 by $\Phi_\rho^t(\pi(x)) =\pi(\rho^{\exp(tX)}(x))$.
It is well-defined since
 $\exp(tX)$ commutes with any element of $SO(n-1)$.
We call the flow $\Phi_\rho$ {\it the flow} induced by $\rho$.

For $\rho \in \cA_*(M_\Gamma)$,
 we define vector fields $Y_1^\rho, \dots, Y_{n-1}^\rho$ on $M_\Gamma$
 by $Y_i^\rho(x)=(d/dt)\rho^{\exp tY_i}(x)|_{t=0}$.
\begin{lemma}
\label{lemma:reduce 1} 
For any $x \in M_\Gamma$ and $m \in SO(n-1)$,
\begin{equation*}
(D\pi(Y_1^\rho(x \cdot m)),\cdots, D\pi(Y_{n-1}^\rho(x \cdot m))
 =(D\pi(Y_1^\rho(x)),\cdots, D\pi(Y_{n-1}^\rho(x))) \cdot m.
\end{equation*}
\end{lemma}
\begin{proof}
For any $x \in M_\Gamma$, $t \in \RR$, and $m \in SO(n-1)$,
\begin{align*}
\pi \circ \rho(x \cdot m,\exp(t Y_i)) 
 & = \pi(\rho(x, [m \exp(t Y_i)m^{-1}]) \cdot m)\\
 & =\pi \circ \rho(x,\exp (t \cdot \mbox{Ad}_m(Y_i))).
\end{align*}
Hence, the equation (\ref{eqn:Y 2}) implies the lemma.
\end{proof}
By the above lemma,
 we can define a $C^\infty$ subbundle $E^-_\rho$ of $TN_\Gamma$ by
\begin{equation*}
E^-_\rho(\pi(x))=D\pi(\langle Y_1^\rho(x), \dots Y_{n-1}^\rho(x) \rangle).
\end{equation*}
There exists a $C^\infty$ metric $g_\rho$ on $E^-_\rho$ such that
 $(D\pi(Y_1^\rho(x)), \dots, D\pi(Y_{n-1}^\rho(x)))$
 is an orthonormal basis of $E^-_\rho(x)$ with respect to $g_\rho$.
The subbundle $E^-_\rho$ is $D\Phi_\rho$-invariant and
\begin{equation}
\label{eqn:stable bundle}
 \|D\Phi_\rho^t(v)\|_{g_\rho}=e^{-t}\|v\|_{g_\rho} 
\end{equation}
 for any $t \in \RR$ and $v \in E^-_\rho$.

For $i=1,\cdots,n-1$,
 let $Y'_i$ be a vector field on $M_\Gamma$ given by
 $Y_i^+(x)=(d/dt)x\exp(tY'_i)|_{t=0}$.
Similar to the above, 
 we can define a $C^\infty$ subbundle $E^+_{\rho_0}$ of $TN_\Gamma$
 and its $C^\infty$ metric $g^+$ such that
\begin{equation*}
E^+_{\rho_0}(\pi(x))=D\pi(\langle Y_1^+(x), \dots Y_{n-1}^+(x) \rangle)
\end{equation*}
 and $(D\pi(Y_1^+(x)), \dots, D\pi(Y_{n-1}^+(x)))$
 is an orthonormal basis of $E^+_{\rho_0}(x)$ with respect to $g^+$.
The subbundle $E^+_{\rho_0}$ is $D\Phi_{\rho_0}$-invariant and
\begin{equation}
\label{eqn:unstable bundle}
\|D\Phi_{\rho_0}^t(v'))\|_{g^+}=e^t\|v'\|_{g^+} 
\end{equation}
 for any $t \in \RR$ and $v' \in E^+_{\rho_0}$.
%Let $g_\Gamma$ be a $C^\infty$ metric on $TN_\Gamma$
% such that the restriction of $g_\Gamma$ to
% $E^-_{\rho_0}$ and $E^+_{\rho_0}$ are $g_{\rho_0}$
% and $g'$, respectively.
The flow $\Phi_{\rho_0}$ is an Anosov flow
 with the Anosov splitting
 $TN_\Gamma=T\Phi \oplus E^-_{\rho_0} \oplus E^+_{\rho_0}$
 and it is $s$- and $u$-conformal
 with respect to $g_{\rho_0}$ and $g^+$, respectively.
It is known that $E^-_{\rho_0} \oplus E^+_{\rho_0}$ is a
 $\Phi_{\rho_0}$-invariant contact structure.

Since the set of Anosov flows is open in the space of $C^1$ flows,
 the induced flow  $\Phi_\rho$ is Anosov
 if $\rho \in \cA_*(M_\Gamma,P)$ is sufficiently $C^1$-close to $\rho_0$.
In this case, $E^-_\rho$ is the strong stable subbundle of $\Phi_\rho$
 and the Anosov flow $\Phi_\rho$
 is $s$-conformal with respect to $g_\rho$.

\subsection{Reduction to the conjugacy of induced flows}

We reduce Theorem \ref{thm:main thm} to
 the smooth conjugacy problem of the induced flows.
\begin{thm}
\label{thm:reduce} 
Let $\rho$ be a locally free action in $\cA_*(M_\Gamma,P)$.
Suppose that a $C^\infty$ diffeomorphism $h$ of $N_\Gamma$ satisfies
$\Phi_\rho^t \circ h=h \circ \Phi_{\rho_0}^t$ for any $t \in \RR$.
Then, $\rho$ is $C^\infty$-conjugate to the standard
 $P$-action $\rho_0$.
\end{thm}

Let $\Fr E^-_\rho$ be the frame bundle of $E^-_\rho$.
It admits a natural right action of $GL(n-1,\RR)$.
The flow $\Phi_\rho$ induce a flow $\Fr\Phi_\rho$ on $\Fr E^-_\rho$.
Let $OE_\rho^-$ be the orthonormal frame bundle of $(E^-_\rho,g_\rho)$.
We define a map $\psi_\rho:M_\Gamma \ra OE^-_\rho$ by
\begin{equation*}
 \psi_\rho(x)=(D\pi(Y_1^\rho(x)), \dots, D\pi(Y_{n-1}(x))).
\end{equation*}
By Lemma \ref{lemma:reduce 1},
 $OE_\rho^-(y)=\{\psi_\rho(x) \st x \in \pi^{-1}(y)\}$
 for any $y \in N_\Gamma$
 and $\psi_\rho$ is a diffeomorphism from $M_\Gamma$ to $OE_\rho^-$.
By Equation (\ref{eqn:Y 1}), we have
 $\Fr\Phi_\rho^t(\psi(x)) =e^{-t} \psi_\rho(\rho^{\exp(tX)}(x))$.
Hence, we can define a flow $O\Phi_\rho$ on $OE_\rho^-$ by
\begin{equation*}
O\Phi_\rho^t(\psi_\rho(x))=e^t \cdot
 \Fr\Phi_\rho^t(\psi_\rho(x)) =\psi_\rho(\rho^{\exp(tX)}(x)).
\end{equation*}
In particular, the map $\psi_\rho$ is a $C^\infty$ conjugacy between
 $\rho^{\exp(tX)}$ and $O\Phi_\rho^t$.
By Moore's ergodicity theorem,
 the flow $(\rho^{\exp(tX)})_{t \in \RR}$ is topologically transitive.
Hence, so the flow $O\Phi_\rho$ is.

Fix $\rho \in \cA_*(M_\Gamma,P)$
 and suppose that there exists a $C^\infty$ diffeomorphism $h$ of $N_\Gamma$
 such that 
\begin{equation}
\label{eqn:reduce assume}
 \Phi_\rho^t \circ h=h \circ \Phi_{\rho_0}^t 
\end{equation}
 for any $t \in \RR$.
\begin{lemma}
\label{lemma:reduce 2} 
$Dh(E^-_{\rho_0})=E^-_\rho$ and there exists a constant $c_h>0$
 such that
 $\|Dh(v)\|_{g_\rho}=c_h \cdot \|v\|_{g_\Gamma}$
 for any $v \in E^-_{\rho_0}$.
\end{lemma}
\begin{proof}
Recall that the flow $\Phi_{\rho_0}$ is Anosov and
 $E^-_{\rho_0}$ is its strong stable subbundle.
Since $h$ is a $C^\infty$ conjugacy between $\Phi_{\rho_0}$ and $\Phi_\rho$,
 the flow $\Phi_\rho$ is also Anosov
 and its strong stable subbundle is $Dh(E^-_{\rho_0})$. 
By Equation (\ref{eqn:stable bundle}),
 the subbundle $E^-_\rho$ is contained
 in the strong stable subbundle $Dh(E^-_{\rho_0})$.
Since their dimensions are equal,
 we have $Dh(E^-_{\rho_0})=E^-_\rho$.

Let $SE^-_{\rho_0}$ be the unit sphere bundle
 $\{v \in E^-_{\rho_0} \st \|v\|_{g_{\rho_0}}=1\}$ of $E^-_{\rho_0}$
 and $\pi_O:OE^-_{\rho_0} \ra SE^-_{\rho_0}$
 be the projection defined by  $(v_1,\dots,v_{n-1}) \mapsto v_1$.
By Equation (\ref{eqn:stable bundle}) for $\rho_0$,
 we can define a flow $S\Phi_{\rho_0}$ on $SE^-_{\rho_0}$
 by $S\Phi_{\rho_0}^t=e^t D\Phi_{\rho_0}^t$.
Then, $\pi_O \circ O\Phi_{\rho_0}^t=S\Phi_{\rho_0}^t \circ \pi_O$.
Since $O\Phi_{\rho_0}$ is topologically transitive,
 $S\Phi_{\rho_0}$ also is.
Take $v_0 \in SE^-_{\rho_0}$
 such that the orbit $\{S\Phi_{\rho_0}^t(v_0) \st t \in \RR\}$
 is dense in $SE^-_{\rho_0}$.
Put $c_h=\|Dh(v_0)\|_{g_\rho}$.
By Equation (\ref{eqn:stable bundle}),
\begin{align*}
\|Dh \circ S\Phi_{\rho_0}^t(v_0)\|_{g_\rho}
 & = e^t \|Dh \circ D\Phi_{\rho_0}^t(v_0)\|_{g_\rho}\\
 & = e^t \|D\Phi_\rho^t \circ Dh(v_0)\|_{g_\rho}\\
 & = \|Dh(v_0)\|_{g_\rho}\\
 & = c_h
\end{align*}
 for any $t \in \RR$.
It implies that
 $\|Dh(v)\|_{g_\rho}=c_h$ for any $v \in SE^-_{\rho_0}$.
\end{proof}

\begin{prop}
There exists a $C^\infty$ diffeomorphism $H$ of $M_\Gamma$
 such that $H \circ \rho_0^g=\rho^g \circ H$
 for any $g \in \{\exp(tX)m \st t \in \RR, m \in SO(n-1)\}$.
\end{prop}
\begin{proof}
Let $\Fr h$ be the lift of $h$ to $\Fr E^-_{\rho_0}$.
By the above lemma, 
 we can define a diffeomorphism $Oh:OE^-_{\rho_0} \ra OE^-_\rho$ 
 by $Oh=c_h^{-1}\Fr h$.
Then, $O\Phi_\rho^t \circ Oh=Oh \circ O\Phi_{\rho_0}^t$
 for any $t \in \RR$.
Since $\Fr h$ commutes with the action of $SO(n-1)$,
 we have $Oh(z \cdot m)=Oh(z) \cdot m$
 for any $z \in OE^-_{\rho_0}$ and $m \in SO(n-1)$.
Put $H=\psi_\rho^{-1} \circ Oh \circ \psi_{\rho_0}$.
Then, we have 
\begin{gather*}
H \circ \rho_0^{\exp(tX)} = \rho^{\exp(tX)} \circ H,\\
H(x \cdot m)=H(x) \cdot m
\end{gather*}
 for any $t \in \RR$ and any $m \in SO(n-1)$.
\end{proof}

The proof of  Theorem \ref{thm:reduce} will finish
 once we show the following
\begin{prop}
There exists an automorphism $\theta$ of $P$
 such that $\rho(H(x),\theta(g))=\rho_0(x,g)$
 for any $x \in M_\Gamma$ and $g \in P$.
\end{prop}
\begin{proof}
Since $\rho^{\exp (tX)} \circ H =H \circ \rho_0^{\exp (tX)}$
 for any $t \in \RR$, we have
\begin{align}
\label{eqn:reduce H 1}
D\rho^{\exp(tX)}(DH(Y_j^{\rho_0}(x)))
 &= DH(D\rho_0^{\exp(tX)}(Y_j^{\rho_0}(x))) \\
 &= e^{-t}DH(Y_j^{\rho_0}(\rho_0^{tX}(x))). \nonumber
\end{align}
We also have
\begin{align*}
DH(\langle Y_1^{\rho_0}(x),\cdots,Y_{n-1}^{\rho_0}(x) \rangle)
 & = DH(\{v \in T_{x} M_\Gamma \st
   \lim_{t \ra +\infty} \|D\rho_0^{\exp(tX)}(v)\|=0\})\\
 & = \{v' \in T_{H(x)} M_\Gamma \st
   \lim_{t \ra +\infty} \|D\rho^{\exp(tX)}(v')\|=0\}\\
 & \supset \langle Y_1^\rho(H(x)),\cdots,Y_{n-1}^\rho(H(x)) \rangle
\end{align*}
 for any $x \in M_\Gamma$.
In particular,
\begin{align*}
DH(\langle Y_1^{\rho_0}(x),\cdots,Y_{n-1}^{\rho_0}(x) \rangle)
 & = \langle Y_1^\rho(H(x)),\cdots,Y_{n-1}^\rho(H(x)) \rangle.
\end{align*}
Since the flow $(\rho_0^{\exp(tX)})_{t \in \RR}$ is topologically transitive,
 there exists $x_0 \in M_\Gamma$ such that
 $\{\rho_0^{tX}(x_0) \st t \in \RR\}$ is a dense subset of $M_\Gamma$.
Let $b=(b_{ij})_{i,j=1,\cdots,n-1}$ be the square matrix
 given by $Y_j^\rho(H(x_0))=\sum_{i=1}^{n-1} b_{ij} DH(Y_i^{\rho_0}(x_0))$.
Remark that it is an invertible matrix.
For any $t \in \RR$, we have
\begin{align*}
Y_j^\rho(H(\rho_0^{\exp(tX)}(x_0)))
 & = Y_j^\rho(\rho^{\exp(tX)}(H(x_0)))\\
 & = e^t D\rho^{\exp(tX)}(Y_j^\rho(H(x_0)))\\
 & =\sum_{i=1}^{n-1}b_{ij} DH(Y_i^{\rho_0}(\rho_0^{\exp(tX)}(x_0))).
\end{align*}
Since the orbit of $x_0$ is dense,
$Y_j^\rho(H(x))
  =\sum_{i=1}^{n-1} b_{ij}DH(Y_i^{\rho_0}(x))$
 for any $x \in M_\Gamma$.
In particular,
\begin{equation}
\label{eqn:reduce H 2}
DH(Y_1^{\rho_0},\dots,Y_{n-1}^{\rho_0})
 = (Y_1^\rho,\dots,Y_{n-1}^\rho) \cdot b^{-1}
\end{equation}

Recall that
\begin{align*}
D\rho_0^m(Y_1^{\rho_0},\dots,Y_{n-1}^{\rho_0})
 & = (Y_1^{\rho_0},\dots,Y_{n-1}^{\rho_0}) \cdot m,\\
D\rho^m(Y_1^\rho,\dots,Y_{n-1}^\rho)
 & = (Y_1^\rho,\dots,Y_{n-1}^\rho) \cdot m
\end{align*}
 for any $m \in SO(n-1)$.
Since $\rho^m \circ H= H \circ \rho_0^m$ for any $m$,
 Equation (\ref{eqn:reduce H 2}) implies
\begin{equation*}
(Y_1^{\rho_0},\dots, Y_{n-1}^{\rho_0}) \cdot mb^{-1}
 = (Y_1^{\rho_0},\dots, Y_{n-1}^{\rho_0}) \cdot b^{-1}m.
\end{equation*}
Hence, $b$ commutes with any $m \in SO(n-1)$.
It is easy to check that
\begin{itemize}
\item if $n \geq 4$, then there exists $\alpha \in \RR \setminus \{0\}$
 such that $b=\alpha I_{n-1}$, where $I_{n-1}$ is the unit matrix
 of size $(n-1)$,
\item if $n =3$, then there exists $\alpha \in \RR \setminus \{0\}$
 and $m_0 \in SO(2)$ such that such that $b=\alpha m_0$.
\end{itemize}
In each case, $\alpha^{-1}b$ is contained in the center of $SO(n-1)$.

We define a map $\theta:P \ra \SO$ by $\theta(g)=b g b^{-1}$.
Since $\alpha^{-1}b$ is contained in the center of $SO(n-1)$,
 we have $\theta(P)=P$.
In particular, $\theta$ is an automorphism of $P$
 such that
\begin{equation*}
 \theta_*(Y_1,\dots,Y_{n-1})=(Y_1,\dots,Y_{n-1}) \cdot b^{-1},
\end{equation*}
 where $\theta_*$ is the induced automorphism of the Lie algebra of $P$.
By Equation (\ref{eqn:reduce H 2}), 
 $\rho(H(x),\theta(\exp(Y_i))=H(\rho_0(x,\exp(Y_i)))$ 
 for any $x \in M_\Gamma$ and $i=1,\dots,n-1$.
On the other hand,
 $\theta(g')=g'$ and
 $\rho^{g'} \circ H=H \circ \rho_0^{g'}$
 for any $g' \in \{\exp(tX)m \st t \in \RR, m \in SO(n-1)\}$.
Therefore,
 $\rho(H(x),\theta(g)=H(\rho_0(x,g))$ 
 for any $x \in M_\Gamma$ and $g \in P$.
\end{proof}

\subsection{Smooth conjugacy between induced flows}

In this subsection, we show the following theorem.
With Proposition \ref{prop:Palais}
 and Theorem \ref{thm:reduce},
 it completes the proof of the main theorem.
\begin{thm}
\label{thm:conj flow}
If $\rho \in \cA_*(M_\Gamma,P)$ is
 sufficiently $C^\infty$-close to $\rho_0$,
 then there exists a $C^\infty$ diffeomorphism $h$ of $N_\Gamma$
 such that $\Phi_\rho^t \circ h=h \circ \Phi_{\rho_0}^t$ for any $t \in \RR$.
\end{thm}

Choose $\rho \in \cA_*(M_\Gamma,P)$ such that
 $\Phi_\rho$ is an $s$-conformal Anosov flow with respect to $g_\rho$
 and $E^-_\rho$ is transverse to $T\Phi_{\rho_0} \oplus E^+_{\rho_0}$.
By $\cF^{ss}_{\rho_0}$, $\cF^s_{\rho_0}$,
 $\cF^{uu}_{\rho_0}$, $\cF^u_{\rho_0}$,
 we denote the strong stable,
 strong unstable, weak stable, weak unstable foliations
 of $\Phi_{\rho_0}$, respectively.
Similarly, by $\cF^{ss}_\rho$, $\cF^s_\rho$,
 we denote the strong stable and weak stable foliations of $\Phi_\rho$,
 respectively.
Remark that all of them are $C^\infty$ foliations,
 but the strong unstable and weak unstable foliations
 of $\Phi_\rho$ may not be $C^\infty$.

By $X_{\rho_0}$ and $X_\rho$, we denote the vector fields
 generating the flows $\Phi_{\rho_0}$ and $\Phi_\rho$.
Let $\sigma_1:TN_\Gamma \ra E^+_{\rho_0}$
 and $\sigma_2:TN_\Gamma \ra E^-_\rho$
 be the projection with respect to the splittings
 $E^+_{\rho_0} \oplus (T\Phi_\rho \oplus E^-_\rho)$ and
 $E^-_\rho \oplus (T\Phi_{\rho_0} \oplus E^+_{\rho_0})$
 of $TN_\Gamma$, respectively.
Put $X_1=X_{\rho_0}-\sigma_1(X_{\rho_0})$
 and $X_2=X_\rho - \sigma_2(X_\rho)$.
They generates flows $\Psi_1$ and $\Psi_2$ on $N_\Gamma$.
If $\rho$ is sufficiently $C^\infty$-close to $\rho_0$,
 then $\Psi_1$ and $\Psi_2$ are $C^\infty$-close to $\Phi_{\rho_0}$.
Hence, we may assume that they are Anosov flows.
Since $[X_1, E^+_{\rho_0}] \subset E^+_{\rho_0}$
 and $X_1 \in T\Phi_\rho \oplus E^-_\rho$,
 we have
\begin{equation}
\label{eqn:Psi 1}
 \Psi_1^t(x) \in
 \cF^{uu}_{\rho_0}(\Phi_{\rho_0}^t(x)) \cap \cF^s_\rho(x).
\end{equation}
 for any $x \in M_\Gamma$ and $t \in \RR$.
If $\rho$ is sufficiently close to $\rho_0$,
 then $D\Psi_1$ expands $E^+_{\rho_0}$ uniformly.
So, we may assume that $E^+_{\rho_0}$
 is the strong unstable subbundle of $\Psi_1$.
Similarly,
 we may assume 
\begin{equation}
\label{eqn:Psi 2}
\Psi_2^t(x) \in \cF^{ss}_\rho(\Phi_\rho^t(x))\cap \cF^u_{\rho_0}(x)
\end{equation}
 for any $x \in N_\Gamma$ and $t \in \RR$,
 and $E^-_\rho$ is the strong stable subbundle of $\Psi_2$.
Since both $\Psi_1^t(x)$ and $\Psi_2^t(x)$ are
 contained in $\cF^u_{\rho_0}(x) \cap \cF^s_\rho(x)$,
 the orbits of $\Psi_1$ and $\Psi_2$ coincide.
Hence, there exists a $C^\infty$ cocycle over $\Psi_2$ such that
\begin{equation}
\label{eqn:cocycle alpha}
\Psi_2^t(x)=\Psi_1^{\alpha(x,t)}(x) 
\end{equation}
 for any $x \in N_\Gamma$ and $t \in \RR$.
Since each leaf of $\cF^u_{\rho_0}$ is $\Phi_{\rho_0}$-
 and $\Phi_{\rho}$-invariant and
 it is transverse to both $E^-_{\rho_0}$ and $E^-_{\rho}$,
 we have
\begin{equation}
\label{eqn:cocycle alpha 1}
\det D\Psi_1^{\alpha(x,T)}|_{E^-_{\rho_0}(x)} 
 =\det D\Psi_2^T|_{E^-_\rho(x)}
\end{equation}
 for any $(x,T) \in N_\Gamma \times \RR$
 satisfying $\Psi_2(x)^T=x$.

By Proposition \ref{prop:fol conj},
 there exist a homeomorphism $h_1$ of $N_\Gamma$
 such that $\Psi_1^t \circ h_1=h_1 \circ \Phi_{\rho_0}^t$
 for any $t \in \RR$
 and $h_1(\cF^{uu}_{\rho_0}(x))=\cF^{uu}_{\rho_0}(x)$
 for any $x \in N_\Gamma$.
Since $h_1$ preserves each leaf of $\cF^u_{\rho_0}$,
 we have
\begin{equation}
\label{eqn:derivative 1}
 \det D\Psi_1^{T'}|_{E^-_\rho(x)} 
 =\det D\Phi_{\rho_0}^{T'}|_{E^-_{\rho_0}(h_1^{-1}(x))}
 =e^{-(n-1)T'}
\end{equation}
 for any $(x,T') \in N_\Gamma \times \RR$
 satisfying $\Psi_1^{T'}(x)=x$.

By Proposition \ref{prop:fol conj} again,
 there exist a homeomorphism $h_2$ of $N_\Gamma$
 such that $\Psi_2^t \circ h_2=h_2 \circ \Phi_\rho^t$
 for any $t \in \RR$
 and $h(\cF^{ss}_\rho(x))=\cF^{ss}_\rho(x)$ for any $x \in N_\Gamma$.
Since $\cF^u_{\rho_0}$ is a transversely conformal foliation,
 $\Psi_2$ is $s$-conformal.
By Theorem \ref{thm:Llave},
 the restriction of $h_2$ to each leaf of $\cF^{ss}_\rho$
 is smooth.
Hence, 
\begin{equation}
\label{eqn:derivative 2}
 \det D\Psi_2^T|_{E^-_\rho(x)} 
 =\det D\Phi_\rho^T|_{E^-_\rho(h_2^{-1}(x))}
 =e^{-(n-1)T}
\end{equation}
 for any $(x,T) \in N_\Gamma \times \RR$
 satisfying $\Psi_2^T(x)=x$.

By Equations (\ref{eqn:cocycle alpha 1}),
 (\ref{eqn:derivative 1}), and (\ref{eqn:derivative 2}), 
 we have $\alpha(x,T)-T=0$ for any $(x,T) \in N_\Gamma \times \RR$
 satisfying $\Psi_2^T(x)=x$.
By The $C^\infty$ Livschitz Theorem,
 there exists a $C^\infty$ function $\beta$ on $N_\Gamma$ such that
\begin{equation*}
\alpha(x,t)-t=\beta(\Psi_2^t(x))-\beta(x) 
\end{equation*}
 for any $x \in N_\Gamma$ and $t \in \RR$.
We define a map $h_3:N_\Gamma \ra N_\Gamma$ by $h_3(x)=\Psi_2^{-\beta(x)}$.
Remark that if $\rho$ is sufficiently $C^\infty$-close to $\rho_0$,
 then $\alpha$ is $C^\infty$-close to zero,
 and hence, we can choose $\beta$ so that it is $C^\infty$-close to $0$.
So, we may assume that $h_3$ is a $C^\infty$ diffeomorphism
 sufficiently $C^\infty$-close to the identity map.
Since $\Phi_{\rho_0}$ is a contact Anosov flow
 and the set of contact structures is open 
 in the space of $C^1$ hyperplane fields,
 we also may assume that
 $Dh_3(E^-_\rho) \oplus E^+_{\rho_0}$ is a contact structure.

By Equation (\ref{eqn:cocycle alpha}), we have
$\Psi_1^t \circ h_3 = h_3 \circ \Psi_2^t$.
In particular, $Dh_3(E^-_\rho)$ is the strong stable subbundle of $\Psi_1$.
Since $E^+_{\rho_0}$ is the strong unstable subbundle of $\Psi_1$,
 the flow $\Psi_1$ is a contact Anosov flow.
Since $\Psi_1$ is $s$-conformal,
 it is also $u$-conformal by Proposition \ref{prop:contact Anosov}.
By Theorem \ref{thm:Llave},
 $h_1$ is a $C^\infty$ diffeomorphism.

Since $\cF^s_{\rho_0}$ is a transversely conformal foliation,
 $\cF^s_\rho=h_1(\cF^s_{\rho_0})$ also is.
It implies that $\Phi_\rho$ and $\Psi_2$ are $u$-conformal.
Since $\Phi_\rho$ and $\Psi_2$ are $s$-conformal,
 Theorem \ref{thm:Llave} implies
 that $h_2$ is a $C^\infty$ diffeomorphism.
Now, we put $h=h_2^{-1} \circ h_3^{-1} \circ h_1$.
Then, $h$ is a $C^\infty$ diffeomorphism
 and $\Phi_\rho^t \circ h= h \circ \Phi_{\rho_0}^t$ for any $t \in \RR$.
The proof of Theorem \ref{thm:conj flow} is completed.

%
%  References
%

\end{document}